\documentclass{amsart}   	

\usepackage{amsmath}
\usepackage{amssymb}
\usepackage{amsfonts}
\usepackage{hyperref}
\usepackage{tikz-cd}
\usepackage{amscd}
\usepackage[all]{xy}
\usepackage{enumerate}

\def\Z{{\mathbb Z}} 
\def\Q{{\mathbb Q}} 
 
\def\C{{\mathbb C}} 
 
\def\P{{\mathbb P}}

\def\V{{\mathbb V}}

\def\cC{{\mathcal C}}

\def\cD{{\mathcal D}}
\def\d{{\mathfrak{d}}}

\def\Y{{\mathcal{Y}}}
\def\bG{{\mathbb{G}}}
\def\cG{{\mathcal{G}}}
\def\g{{\mathfrak{g}}}
\def\hyp{{\mathrm{hyp}}}
\def\H{{\mathcal H}}
\def\h{{\mathfrak{h}}}

\def\L{{\mathbb{L}}}
\def\M{{\mathcal M}}

\def\O{{\mathcal{O}}}
\def\cP{{\mathcal P}}
\def\p{{\mathfrak{p}}}

\def\r{{\mathfrak{r}}}

\def\T{{\mathcal T}}
\def\t{{\mathfrak t}}
\def\U{{\mathcal U}}
\def\u{{\mathfrak{u}}}
\def\V{{\mathbb{V}}}
\def\cV{{\mathcal V}}
\def\v{{\mathfrak v}}
\def\X{{\mathcal X}}
\def\Z{{\mathbb{Z}}}

\def\D{{\Delta}} 
\def\G{{\Gamma}}

\def\Ql{{\Q_\ell}}

\def\bs{\backslash} 

\def\ab{\mathrm{ab}}
\def\an{\mathrm{an}}

\def\hyp{\mathrm{hyp}}

\def\orb{\mathrm{orb}} 
\def\out{\mathrm{out}}
 
\def\alg{\mathrm{alg}} 
\def\un{\mathrm{un}}
\def\sp{\mathfrak{sp}}

\newcommand{\Gr}{\operatorname{Gr}}
\newcommand{\Hom}{\operatorname{Hom}}
\newcommand{\Inn}{\operatorname{Inn}}
\newcommand{\Lie}{\operatorname{Lie}}
\newcommand{\Out}{\operatorname{Out}}
\newcommand{\Der}{\operatorname{Der}}

\newcommand{\cts}{\operatorname{cts}}
\newcommand\im{\operatorname{im}} 

\newcommand\Aut{\operatorname{Aut}} 
\newcommand\Diff{\operatorname{Diff}}

\newcommand\Sp{\operatorname{Sp}} 

\newtheorem{theorem}{Theorem}[section]
\newtheorem{lemma}[theorem]{Lemma}
\newtheorem{proposition}[theorem]{Proposition}
\newtheorem{corollary}[theorem]{Corollary}
\newtheorem{bigtheorem}{Theorem}
\newtheorem{bigcorollary}[bigtheorem]{Corollary}

\theoremstyle{definition}

\newtheorem{example}[theorem]{Example}

\theoremstyle{remark}
\newtheorem{remark}[theorem]{Remark}

\begin{document}

\title[On the Birman exact sequence of the subgps of MCG of genus three]{On the Birman exact sequence of the subgroups of the mapping class group of genus three
}
\author{
Ma Luo \and Tatsunari Watanabe
}
\address{School of Mathematical Sciences,  Key Laboratory of MEA (Ministry of Education), Shanghai Key Laboratory of PMMP,  East China Normal University, Shanghai}
\email{mluo@math.ecnu.edu.cn}
\address{Mathematics Department, Embry-Riddle Aeronautical University, Prescott}
\email{watanabt@erau.edu}

\thanks{The first author is supported in part by National Natural Science Foundation of China (No. 12201217), by Shanghai Pilot Program for Basic Research, and by Science and Technology Commission of Shanghai Municipality (No. 22DZ2229014).}

\begin{abstract}
\smallskip
We prove that for any finite index subgroup of the mapping class group containing the Johnson subgroup, the profinite Birman exact sequence does not split in genus $g\ge 3$, extending prior results \cite{hain_rational, wat_rk} for $g\ge 4$. 
For the Torelli group, we prove that the graded Lie algebra version of the Birman exact sequence admits no section with symplectic equivariance, extending Hain's result \cite{hain_inf} from $g\ge4$ to $g=3$. These results are deduced by our main tool, relative completion, with the help of Hodge theory and representation theory of symplectic groups, along with explicit structural obstructions coming from hyperelliptic mapping class groups \cite{wat_hyp_univ}.
\end{abstract}
\maketitle

\setcounter{tocdepth}{1}
\tableofcontents
\section{Introduction}
Let $S_{g,n}$ be an oriented topological surface of genus $g$ with $n$ distinct marked points. The mapping class group $\G_{g,n}$ of $S_{g,n}$ is the group of isotopy classes of orientation-preserving diffeomorphisms of $S_{g,n}$ fixing the $n$ marked points pointwise. There is a natural symplectic representation of $\G_{g,n}$, which we denote by $\rho:\G_{g,n}\to\Sp(H_1(S_g;\Z))$. The kernel of $\rho$ is called the Torelli group, denoted by $T_{g,n}$. 
The Birman exact sequence for the mapping class group 
\begin{equation}\label{birman seq}
1 \to \pi_1(S_g) \to \Gamma_{g,1} \to \Gamma_g \to 1
\end{equation}
is known to not split for $g \geq 2$ (see \cite[Cor.~5.11]{farb_marg}). It follows from a result of Hain in \cite{hain_rational} that the profinite Birman exact sequence
\begin{equation}\label{profinite birman seq}
1 \to \widehat{\pi_1(S_g)} \to \widehat{\Gamma_{g,1}} \to \widehat{\Gamma_g} \to 1
\end{equation}
also does not split for $g \geq 4$ (see also \cite[Thm.~1]{wat_rk}). Furthermore, Chen proved in \cite{chen_torelli} that the Birman exact sequence for the Torelli group 
\begin{equation}\label{birman torelli}
1 \to \pi_1(S_g) \to T_{g,1} \to T_g \to 1
\end{equation}
does not split for $g \geq 4$. Chen and Salter \cite{chen_salter} also proved that the Birman exact sequence does not split for any finite index subgroup of the mapping class group when $g \geq 4$.

In this paper, we primarily investigate the nonsplitting phenomenon in the profinite and graded Lie algebra analogs of the Birman exact sequence for certain subgroups of mapping class groups. Notably, the nonsplitting of these profinite variants implies the nonsplitting of the original Birman exact sequence for the corresponding discrete subgroups.

More precisely, we study the profinite Birman exact sequence for certain finite index subgroups of the mapping class group and the Torelli group for the cases $g \geq 3$. Denote the {\it Johnson subgroup} of $\G_g$ by $K_g$. It is the kernel of the Johnson homomorphism (defined in \cite{john_ab}) $\tau_g: T_g\to \Lambda^3_0H_\Z:=(\Lambda^3H_\Z)/\theta\wedge H_\Z$, where $H_\Z=H_1(S_g;\Z)$ and $\theta \in \Lambda^2 H_\Z$ is the polarization, and it is generated by Dehn twists about separating simple closed curves. For a subgroup $\G$ of $\G_g$, pulling back the Birman exact sequence (\ref{birman seq}) for $\G_g$ along the inclusion $\G\hookrightarrow \G_g$, we obtain the Birman exact sequence for $\G$:
\begin{equation}\label{birman for subgroup}
1\to \pi_1(S_g)\to \G'\to \G\to 1.
\end{equation}
Furthermore, applying profinite completion, we obtain the profinite Birman exact sequence for the profinite completion $\widehat{\G}$ of $\G$:
\begin{equation}\label{profinite birman for subgroup}
1\to \widehat{\pi_1(S_g)}\to \widehat{\G'}\to \widehat{\G}\to 1.
\end{equation}
\begin{bigtheorem}\label{main thm for subgrp of mg} 
Suppose that $\G$ is a finite-index subgroup of $\G_g$ containing $K_g$. 
If $g\geq 3$, the profinite Birman exact sequence (\ref{profinite birman for subgroup}) for $\widehat{\G}$ does not split. Consequently, the Birman exact sequence (\ref{birman for subgroup}) for $\G$ does not split for $g\geq 3$. 
\end{bigtheorem}
For $g \geq 3$, examples of these finite-index subgroups of $\Gamma_g$ containing $K_g$ include the fundamental groups of moduli spaces of Riemann surfaces of genus $g$ with an abelian level structure, Prym level structure (defined in \cite{bog_pik_prym}), or a root of their canonical bundle.

For the Torelli group, we consider the Malcev Lie algebra $\t_{g,n}$ of $T_{g,n}$ (see \S 4).  Hain proved in \cite{hain_inf} that $\t_{g,n}$ is finitely presented for $g \geq 3$. The Birman exact sequence (\ref{birman torelli}) for the Torelli group induces an exact sequence of graded Lie algebras associated to the lower central series $L^\bullet\t_{g,n}$:
\begin{equation}\label{graded malcev seq}
0 \to \Gr^\bullet_L\p \to \Gr^\bullet_L\t_{g,1} \to \Gr^\bullet_L\t_g \to 0,
\end{equation}
where $\p$ is the Malcev Lie algebra of $\pi_1(S_g)$. Johnson's computation of $H_1(T_g)$ in \cite{john_torelli3} implies that each graded piece $\Gr^m_L\t_{g,n}$ is an $\Sp_{2g}(\Q)$-representation. Our second main result concerns the question of whether the sequence \eqref{graded malcev seq} splits for $g=3$.

\begin{bigtheorem}\label{main thm for tg} 
If $g\geq 3$, the projection $\Gr^\bullet_L\t_{g,1} \to \Gr^\bullet_L\t_g$ has no $\Sp_{2g}(\Q)$-equivariant graded Lie algebra section. 
\end{bigtheorem}
For $g\geq 4$, Theorem \ref{main thm for tg} follows from \cite[Prop.~11.2]{hain_inf}. As a consequence, we have the following result.
\begin{bigcorollary} \label{no sp on ab of tg}
For $g\geq3$, the surjection $T_{g,1}\to T_g$ has no section that induces an $\Sp_{2g}(\Q)$-equivariant map on the abelianization over $\Q$. 
\end{bigcorollary}
Our main tool is the relative completion of a discrete group. Fix an involution $\sigma$ of $S_g$, and denote the hyperelliptic mapping class group of $S_g$ determined by $\sigma$ by $\D_g$. Let $\D_{g,n}$ be the fiber product $\D_g\times_{\G_g}\G_{g,n}$. Let $\cG_{g,n}$ and $\cD_{g,n}$ be the relative completions of $\G_{g,n}$ and $\D_{g,n}$ and $\U_{g,n}$ and $\cV_{g,n}$ be their unipotent radicals, respectively (see \S 5 for the definition). There are exact sequences:
$$
1\to \U_{g,n}\to \cG_{g,n}\to \Sp(H)\to 1
$$
and 
$$
1\to \cV_{g,n}\to \cD_{g,n}\to \Sp(H)\to 1,
$$
where $H:=H_1(S_g;\Q)$.
By a result of Hain in \cite{hain_hodge}, the Lie algebras $\u_{g,n}$ and $\v_{g,n}$ of $\U_{g,n}$ and $\cV_{g,n}$, respectively, admit natural weight filtrations $W_\bullet\u_{g,n}$ and $W_\bullet\v_{g,n}$ from Hodge theory. Hain proved in \cite{hain_inf} that the Lie algebra $\t_{g,n}$ admits a $\Q$-mixed Hodge structure lifting that of $\u_{g,n}$. The canonical map  $T_{g,n}\to \U_{g,n}$ produced by relative completion induces a morphism of MHSs $\t_{g,n}\to \u_{g,n}$, which makes the diagram
$$
\xymatrix@R=1em@C=2em{\t_{g,1}\ar[r]\ar[d]&\u_{g,1}\ar[d]\\
\t_g\ar[r]&\u_g
}
$$
commute. Furthermore, there is a fiber product diagram of graded associated Lie algebras
$$
\xymatrix@R=1em@C=2em{
\Gr^W_\bullet\u_{g,1}\ar[d]&\ar[l]\Gr^W_\bullet\v_{g,1}\ar[d]\\
\Gr^W_\bullet\u_{g}&\ar[l]\Gr^W_\bullet\v_g.
}
$$

For $g\geq 4$, the obstruction for the splitting of the sequence (\ref{profinite birman seq}) essentially comes from the fact that as an $\Sp_{2g}(\Q)$ representation, $\Lambda^2\left(\Lambda^3_0H\right)$ contains a copy of the irreducible representation $\Lambda^2_0H:=\Lambda^2H/\langle \theta\rangle$. There are $\Sp(H)$-equivariant isomorphisms
$$
\Gr^W_{-1}\u_{g,1}\cong H\oplus \Lambda^3_0H \,\,\,\,\text{ and } \Gr^W_{-2}\p\cong \Lambda^2_0H,
$$
and the results in \cite[\S 11]{hain_inf} show that there exists a pair of vectors in $\Lambda^3_0H\subset\Gr^W_{-1}\u_{g,1}$ that maps to $\Gr^W_{-2}\p$ via the bracket operation of $\Gr^W_\bullet\u_{g,1}$.  
The issue for the case $g=3$ is that the representation $\Lambda^2\left(\Lambda^3_0H\right)$ does not contain the representation $\Lambda^2_0H$. In order to overcome this issue, we use variants of the results in \cite{wat_hyp_univ} stating that the projection $\Gr^W_\bullet \v_{g,1} \to \Gr^W_\bullet \v_g$ does not admit an $\Sp_{2g}(\Q)$-equivariant section.

 The outline of the proof of Theorem \ref{main thm for tg} is as follows. The lower central series $L^\bullet \t_{g,n}$ \textcolor{black}{essentially} agrees with the weight filtration $W_\bullet \t_{g,n}$.
 An $\Sp_{2g}(\Q)$-equivariant graded Lie algebra section for $\Gr^W_\bullet \t_{g,1} \to \Gr^W_\bullet \t_g$ induces an $\Sp_{2g}(\Q)$-equivariant section for $\Gr^W_\bullet \u_{g,1} \to \Gr^W_\bullet \u_g$, and consequently for $\Gr^W_\bullet \v_{g,1} \to \Gr^W_\bullet \v_g$, which admits no such section.\\

\noindent
\emph{Acknowledgement:} The authors would like to thank Nick Salter for communicating with us the result that he collaborated with Lei Chen about the nonsplitting of the Birman exact sequence for any finite index subgroup of the mapping class group in $g\ge 4$ \cite{chen_salter}. Our results partially extend theirs in certain specific cases, and our method is completely different from theirs.

\section{The mapping class groups and hyperelliptic mapping class groups}
Suppose that $g$ and $n$ are nonnegative integers satisfying $2g-2 +n >0$. Let $S_{g,n}$ be a compact oriented surface of genus $g$ with $n$ distinct marked points $\{x_1,\ldots, x_n\}$. The mapping class group $\G_{g,n}$ is defined as the group of isotopy classes of orientation-preserving diffeomorphisms of $S_{g,n}$ that fix the $n$ marked points pointwise:
$$
\G_{g,n}:=\mathrm{Diff}^+(S_{g,n}, \{x_1, \ldots, x_n\})/\sim,
$$
where $\sim$ denotes the isotopy relation. For the case $n=0$, denote $S_{g,0}$ and $\G_{g,0}$ by $S_g$ and $\G_g$, respectively.

Consider an orientation-preserving diffeomorphism, denoted by $\sigma$, of order 2, fixing exactly $2g+2$ points. We call it a hyperelliptic involution $\sigma$ and the $2g+2$ fixed points are the Weierstrass points of $S_g$.
By the Riemann-Hurwitz formula, $S_g/\langle \sigma\rangle$ is a sphere, and hence it follows that all such involutions are conjugate in $\Diff^+(S_g)$. 
The hyperelliptic mapping class group $\D_g$ is defined as the group of isotopy classes of elements in the centralizer of $\sigma$ in $\Diff^+(S_g)$. We have the following remarkable result by Birman and Hilden. 
\begin{theorem}[Birman-Hilden \cite{birman-hilden}] 
The natural homomorphism $\D_g \to \G_g$ is injective. Its image is the 
centralizer of the isotopy class of $\sigma$ in $\G_g$. 
\end{theorem} 
When $g=1, 2$, $\D_g = \G_g$. 
Define $\Delta_{g,n}$ as the fiber product of $\Delta_g$ and $\G_{g,n}$ over $\G_g$:
$$\Delta_g\times_{\G_g}\G_{g,n},$$
where the map $\G_{g,n}\to \G_g$ is the projection obtained by forgetting the $n$ marked points. 
\subsection{Symplectic representation of the mapping class group} For a commutative ring \( R \), let \( H_R \) denote the first homology group \( H_1(S_g; R) \) of \( S_g \). When $R=\Q$, we simply write $H$ for $H_\Q$ in this paper.  This group \( H_R \) is a free \( R \)-module of rank \( 2g \). The algebraic intersection pairing \( \langle \cdot, \cdot \rangle \) on \( H_R \) is a nondegenerate alternating form. Define \( \Sp(H_R) \) as the automorphism group \( \Aut(H_R, \langle \cdot, \cdot \rangle) \) of \( H_R \) that preserves the pairing \( \langle \cdot, \cdot \rangle \).

The mapping class group \( \Gamma_{g,n} \) acts on $H_\Z$, preserving the pairing \( \langle \cdot, \cdot \rangle \). Thus, there exists a representation
\[
\rho: \Gamma_{g,n} \to \Sp(H_{\mathbb{Z}}).
\]
It is well known that the representation $\rho$ is surjective (see \cite[Thm.~6.4]{farb_marg}). 

\subsection{Level structure}
For each integer $m\ge 0$, the principal congruence subgroup $\Sp(H_\Z)[m]$ of $\Sp(H_\Z)$ of level $m$ is the kernel of the reduction mod $m$ mapping:
$$
\Sp(H_\Z)[m] = \ker\{\Sp(H_\Z) \to \Sp(H_{\Z/m})\}.
$$
Define the level $m$ subgroup of $\G_{g,n}$ to be
the kernel of the reduction of $\rho$ mod $m$:
$$
\G_{g,n}[m] = \ker\{\G_{g,n} \to \Sp(H_{\Z/m})\}.
$$
The Torelli group $T_{g,n}$ is the level $0$ subgroup of $\G_{g,n}$:
$$
T_{g,n} = \G_{g,n}[0] = \ker\{\G_{g,n} \to \Sp(H_{\Z})\}.
$$
Since $\Sp(H_\Z)[m]$ is torsion free for all $m\ge 3$, it follows that $\G_{g,n}[m]$
is torsion free for all $m\ge 3$.
For $g \geq 1$, $T_{g,n}$ are torsion free (e.g. \cite[Thm.~6.12]{farb_marg}, and for $g\geq 3$, $T_g$ is finitely generated (e.g. see \cite{john}). Mess showed in \cite{mess} that $T_2$ is a countably generated free group. 

\subsection{Hyperelliptic Torelli group} The hyperelliptic Torelli group, denoted by $T\Delta_g$, is defined as the intersection of $\D_g$ and $T_g$:
$$
T\Delta_g:= \D_g\cap T_g.
$$
We have $T\Delta_2 = T_2$, and it is not known for $g\geq 3$, whether $T\Delta_g$ is finitely generated or not.

\section{The universal curves and Torelli space}

Suppose that $2g-2+n >0$. By a complex curve or simply a curve, we mean a Riemann surface in this paper. Let $\X_{g,n}$ be the Teichmüller space of marked, $n$-pointed, compact curves of genus $g$. As a set $\X_{g,n}$ is given by
$$
\left\{
\parbox{2.15in}{orientation preserving diffeomorphisms $f: S_{g,n} \to C$
to a curve}
\right\}
\bigg/
\left\{\parbox{2.15in}{isotopies constant on the $n$ distinct points of $S_{g,n}$}\right\}.
$$
It is a contractible complex manifold of dimension $3g-3+n$. There is an action of $\G_{g,n}$ on $\X_{g,n}$ as a group of biholomorphisms:
$$
\phi: f\mapsto f\circ \phi^{-1},\,\,\,\, \phi \in \G_{g,n},\,\,f\in \X_{g,n}.
$$
It is well known that this action is properly discontinuous  and virtually free. The isotropy group of the class $[S_{g,n}\to (C;x_1, \ldots, x_n)]$ is isomorphic to the automorphism group of $C$ fixing $\{x_1, \ldots, x_n\}$ pointwise. 

Suppose that $m$ is a non-negative integer. A {\it level $m$ structure} on a
curve $C$ is a symplectic basis (with respect to the
intersection pairing) of $H_1(C;\Z/m)$. The moduli space of $n$-pointed smooth
projective curves $g$ with a level $m$ structure is the quotient of
$\X_{g,n}$ by the level $m$ subgroup of $\G_{g,n}$:
$$
\M^\mathrm{an}_{g,n}[m] = \big(\G_{g,n}[m]\bs \X_{g,n}\big)^\orb.
$$
 When $m=1$, $\M_{g,n}^\an:=\M_{g,n}^\an[1]$ is the moduli space of $n$-pointed curves of genus $g$. 
It will be regarded as a complex analytic orbifold.  It is a model of the
classifying space of $\G_{g,n}[m]$ and there is a unique conjugacy class of isomorphisms between its orbifold fundamental group and  $\G_{g,n}[m]$:
$$
\pi_1^\orb(\M_{g,n}^\an[m])\cong \G_{g,n}[m].
$$ 
When the group $\G_{g,n}[m]$ is torsion
free, $\M_{g,n}^\an[m]$ is a smooth complex variety and also a fine moduli space for
$n$-pointed smooth projective curves of genus $g$ with a level $m$ structure.
This occurs, for example, when $m\ge 3$, $m=0$, or when $n>2g+2$.

Note that $\M_{g,n}^\an[0]$ is the quotient of
the Teichm\"uller space by the Torelli group $T_{g,n}$. It is known as {\it Torelli
space}, denoted by $\T_{g,n}$, and the fundamental group is isomorphic to $T_{g,n}$.
\subsection{Moduli space of hyperelliptic curves}
Suppose that $g\ge 2$. We continue the notation of the previous sections.
Denote by $\Y_g$ the set of points of
$\X_g$ fixed by $\sigma$:
$$
\Y_g := \X_g^\sigma.
$$
This is the set of points $[f : S_g \to C]$ of $X_g$ for which $f\sigma f^{-1}
\in \Aut C$. Thus, $\Y_g$ is the set of isotopy classes of framed complex
structures $f : S_g \to C$ on $S_g$ where the isotopy class of $\sigma$ in $\G_g$
contains an automorphism of $C$. It is biholomorphic to the Teichm\"uller space
$\X_{0,2g+2}$, and therefore connected and contractible.

Since there is exactly one conjugacy class of hyperelliptic involutions in
$\G_g$, the locus of hyperelliptic curves in $\X_g$ is
\begin{equation}
\label{union}
\X_g^\hyp = \bigcup_{\phi \in \G_g/(\G_g)_{\Y_g}} \phi(\Y_g).
\end{equation}
The stabilizer $(\G_g)_{\Y_g}$ of $\Y_g$ is easily seen to be $\D_g$. 
Define
$$
\H_g^\an = \big(\D_g \bs \Y_g\big)^\orb.
$$
Furthermore, define
$$
\H_{g,n}^\an = \H_g^\an\times_{\M_{g}^\an}\M_{g,n}^\an.
$$
This will be regarded as an analytic orbifold. Since $\Y_g$ is contractible,
$\H_g^\an$ is a model of the classifying space of $\D_g$, and there is a unique conjugacy class of isomorphisms between its orbifold fundamental group  and  $\Delta_g$:
$$
\pi_1^\orb(\H_g^\an)\cong \Delta_g.
$$ 
Similarly, there is a unique isomorphism, up to conjugation,
$$
\pi_1^\orb(\H_{g,n}^\an)\cong \Delta_{g,n}.
$$
\subsection{Universal curves} Associated to the orbifold $\M^\an_{g,n}[m]$, there exists a universal curve
$$
\pi \colon \cC_{g,n}^\an[m] \to \M^\an_{g,n}[m],
$$
where $\mathcal{C}^\an_{g,n}[m]$ is the universal family of curves over $\mathcal{M}^\an_{g,n}[m]$. Let $C$ be a curve of genus $g$. The fiber of $\pi$ over a point $[C] \in \mathcal{M}^\an_{g,n}[m]$ is isomorphic to $C$, and there is a homotopy exact sequence of fundamental groups:
\[
1 \to \pi_1(C, x) \to \pi_1^{\mathrm{orb}}(\mathcal{C}^\an_{g,n}[m], x) \to \pi_1^{\mathrm{orb}}(\mathcal{M}^\an_{g,n}[m], [C]) \to 1,
\]
where $x$ is a point in $C$. Note that $\pi_1(C, x)\cong \pi_1(S_g)$.

\subsection{Moduli stack of curves}
Denote by $\M_{g,n/\C}[m]$ the moduli stack of smooth, projective $n$-pointed curves of genus $g$ with level $m$ structure over $\C$. For the definition of stacks, see \cite{DM}. For $m \geq 3$, $\M_{g,n/\C}[m]$ is a smooth projective variety over $\C$, and its underlying orbifold is $\M_{g,n}^\an[m]$.

There exists a universal curve
\[
\pi \colon \cC_{g,n/\C}[m] \to \M_{g,n/\C}[m]
\]
with the following universal property: if $f \colon X \to B$ is a proper, smooth family of genus $g$ curves with level $m$ structure and $n$ disjoint sections, then there is a unique morphism $\phi \colon B \to \M_{g,n/\C}[m]$ such that $f$ is isomorphic to the pullback family $\phi^*(\pi)$.

The algebraic fundamental groups of $\M_{g,n/\C}[m]$ and $\cC_{g,n/\C}[m]$ are denoted by $\pi_1^\alg(\M_{g,n/\C}[m])$ and $\pi_1^\alg(\cC_{g,n/\C}[m])$, respectively.
There exists a unique isomorphism, up to conjugacy:
\[
\pi_1^\alg(\M_{g,n/\C}[m]) \cong \widehat{\G_{g,n}[m]},
\]
where $\widehat{\G_{g,n}[m]}$ denotes the profinite completion of the mapping class group $\Gamma_{g,n}[m]$.

Let $[C]$ be a point of $\M_{g,n/\C}[m]$. The universal curve $\pi$ induces an exact sequence of algebraic fundamental groups:  
\[
1 \to \pi_1^\alg(C, x) \to \pi_1^\alg(\cC_{g,n/\C}[m],x) \to \pi_1^\alg(\M_{g,n/\C}[m],[C]) \to 1,
\]
where $x$ is a point in $C$. It is a basic fact that $\pi_1^\alg(C, x)$ is isomorphic to the profinite completion of $\pi_1(S_g)$.

\subsection{Moduli stack of hyperelliptic curves}
Let $\H_{g/\C}$ denote the moduli stack of smooth, projective hyperelliptic curves of genus $g$ over $\C$. There exists a closed immersion 
\[
i:\H_{g/\C} \hookrightarrow \M_{g/\C},
\]
allowing us to view $\H_{g/\C}$ as a closed substack of $\mathcal{M}_{g/\C}$. Define $\H_{g,n/\C}$ as the fiber product of $\H_{g/\C}$ and $\mathcal{M}_{g,n/\C}$ over $\mathcal{M}_{g/\C}$:
\[
\H_{g,n/\C} := \H_{g/\C} \times_{\M_{g/\C}} \M_{g,n/\C}.
\]
The induced closed immersion $\H_{g,n/\C}\hookrightarrow \M_{g,n/\C}$ will be also denoted by $i$. 
The universal hyperelliptic curve 
\[
\pi^{\mathrm{hyp}} : \cC_{\H_{g,n/\C}} \to \H_{g,n/\C}
\]
is defined as the pullback of the universal curve 
$\pi$ to $\H_{g,n/\C}$: 
$$
\xymatrix@R=1em@C=2em{\cC_{\H_{g,n/\C}}\ar[r]\ar[d]_{\pi^\hyp}&\cC_{g,n/\C}\ar[d]^\pi\\
\H_{g,n/\C}\ar[r]^-i&\M_{g,n/\C}.
}
$$
There exists, up to conjugation, a unique isomorphism between $\pi_1^\alg(\H_{g,n/\C})$ and the profinite completion of $\Delta_{g,n}$:
$$
\pi_1^\alg(\H_{g,n/\C})\cong \widehat{\Delta_{g,n}}.
$$ 
\section{The Malcev completion of groups}

\subsection{Malcev completion and Malcev Lie algebra}
Suppose that a discrete group $\Gamma$ is finitely generated, its \textit{Malcev completion} (i.e. \textit{unipotent completion}) is a pro-unipotent group $\Gamma^\un_{/\Q}$ over $\Q$ equipped with a Zariski dense homomorphism $\phi:\Gamma\to\Gamma^\un(\Q)$. It is equivalent to considering its pro-nilpotent Lie algebra $\Lie\Gamma^\un$, which is called the \textit{Malcev Lie algebra associated to $\Gamma$}. This is a special case of our main tool, relative (unipotent) completion, which will be explained in the next section. The standard approach to Malcev completion below follows from \cite[Appendix A]{quillen}. Unless emphasizing, we usually omit the subscript $_{/\Q}$ when it is defined over $\Q$.

The group algebra $\Q\Gamma$ over $\Q$ is naturally a Hopf algebra with coproduct, antipode and augmentation given by
$$\Delta:g\mapsto g\otimes g,\quad i:g\mapsto g^{-1},\quad \epsilon:g\mapsto 1.$$
Note that it is cocommutative but not necessarily commutative, and thus does not correspond to the coordinate ring of a pro-algebraic group. It is natural to consider its (continuous) dual, which is commutative. We define the unipotent completion $\Gamma^{\un}_{/\Q}$ of $\Gamma$, a pro-algebraic group over $\Q$, by its coordinate ring
$$\O(\Gamma^{\un}_{/\Q})=\Hom_{\cts}(\Q\Gamma,\Q):=\varinjlim_n\Hom(\Q\Gamma/I^n,\Q),$$
where we give $\Q\Gamma$ a topology by powers of its augmentation ideal $I:=\ker\epsilon$. 

One can also consider the $I$-adic completion
$$\Q\Gamma^{\wedge}:=\varprojlim_n\Q\Gamma/I^n$$
of $\Q\Gamma$, which is a complete Hopf algebra with (completed) coproduct
$$\Delta:\Q\Gamma^{\wedge}\to\Q\Gamma^{\wedge}\hat{\otimes}\Q\Gamma^{\wedge}$$
induced by the coproduct of $\Q\Gamma$. Denote the completion of $I$ by $I^{\wedge}$. The group of $\Q$-points of the unipotent completion $\Gamma^\un$ is the set of group-like elements
$$\Gamma^\un(\Q)=\{u\in\Q\Gamma^{\wedge} \big| \Delta u=u\otimes u\}\subseteq 1+I^{\wedge}.$$
Its Lie algebra is the set
$$\Lie\Gamma^\un:=\{v\in\Q\Gamma^{\wedge} \big| \Delta v=v\otimes1+1\otimes v\}\subseteq I^{\wedge}.$$
The natural map $\Gamma\to \Q\Gamma^{\wedge}$ induces a natural homomorphism $\Gamma\to\Gamma^\un(\Q)$. 

The exponential and logarithm maps
\begin{center}
    \begin{tikzcd}
    I^{\wedge} \ar[r, shift left, "\exp"] & 1+I^{\wedge} \ar[l, shift left, "\log"]
    \end{tikzcd}
\end{center}
are mutually inverse bijections that restrict to the following bijection
\begin{center}
    \begin{tikzcd}
    \Lie\Gamma^\un \ar[r, shift left, "\exp"] & \Gamma^\un(\Q). \ar[l, shift left, "\log"]
    \end{tikzcd}
\end{center}
This is a group isomorphism if the multiplication of the Lie algebra is given by the Baker--Campbell--Hausdorff (BCH) formula.

\begin{example}[The unipotent completion of Torelli groups]
    The above construction applies to Torelli groups $T_{g,n}$. Denote the unipotent completion of $T_{g,n}$ by $T^\un_{g,n}$, and its Lie algebra by $\t_{g,n}$. The structural presentation of $\t_{g,n}$ will be discussed in \S6.
\end{example}

\begin{example}[The unipotent completion of a fundamental group]
    Take the fundamental group $\pi:=\pi_1(S_g,x)$ of a topological surface $S_g$ of genus $g$ with base point $x\in S_g$. Denote the unipotent completion of $\pi$ by $\cP$, and its Lie algebra by $\p$. The structural presentation of $\p$ will be discussed in \S6. 
    
    When the surface $S_g$ has a complex structure, it becomes a Riemann surface, i.e. a complex curve $C$. Its fundamental group is isomorphic to $\pi$. By abuse of notation, denote by $\cP$ and $\p$ the unipotent completion of the fundamental group $\pi_1(C,x)$ and its Lie algebra, respectively. There is a canonical mixed Hodge structure on $\p$ dependent on the base point $x$, see \S6. 
\end{example}

\subsection{Continuous unipotent completion} Suppose that $k$ is a topological field with characteristic zero. The \textit{continuous unipotent completion} of a topological group $\Gamma$ is analogous to the unipotent completion of a discrete group, with the additional condition that each homomorphism involved being continuous. It is a pro-unipotent group $\Gamma^\un_{/k}$ defined over $k$, equipped with a \textit{continuous} homomorphism $\phi:\Gamma\to\Gamma^\un(k)$. It has the universal property that every \textit{continuous} homomorphism $\Gamma\to U(k)$ with unipotent group $U$ uniquely factors through $\phi$. All of these \textit{continuous} conditions are given by the naturally induced topology on the involved groups from the given topological structures of $\Gamma$ and $k$.

Fix a prime number $\ell$. Take $k=\Q_\ell$. The following result is contained in \cite[Thm. A.6]{hain-matsumoto}.
\begin{theorem}\label{ctsuc}
    If $\Gamma$ is a finitely generated discrete group with profinite completion $\widehat{\Gamma}$, then $$\Gamma^\un_{/\Q}\otimes_{\Q}\Q_{\ell}\quad\text{and}\quad\widehat{\Gamma}^\un_{/\Q_\ell}$$
    are canonically isomorphic as pro-unipotent groups over $\Q_\ell$.
\end{theorem}
\begin{remark}
    This theorem shows that the continuous unipotent completion (over $\Q_\ell$) of a topological group $\widehat{\Gamma}$ can be obtained by starting with the unipotent completion (over $\Q$) of a discrete group $\Gamma$, then base changing from $\Q$ to $\Q_\ell$. Because of this, we will not distinguish the notation for the unipotent completion and its continuous version.
\end{remark}

\begin{example}
    Take the profinite completion $\widehat{T_{g,n}}$ of $T_{g,n}$. Its continuous unipotent completion over $k=\Ql$ is $$\widehat{T_{g,n}}^\un_{/\Ql}=T^\un_{g,n/\Q}\otimes_\Q\Ql$$
    whose Lie algebra is 
    $$\t_{g,n/\Ql}=\t_{g,n}\otimes_\Q\Ql.$$
\end{example}

\begin{example}
    Take the profinite completion $\widehat{\pi_1(S_g)}$ of $\pi_1(S_g)$. Its continuous unipotent completion over $k=\Ql$ is $$\cP_{/\Ql}=\cP_{/\Q}\otimes_\Q\Ql$$
    whose Lie algebra is 
    $$\p_\Ql=\p\otimes_\Q\Ql.$$
\end{example}

\section{Relative completion}
\subsection{Relative completion of a discrete group}
This is a quick review of relative (Malcev/unipotent) completion of discrete groups. More details can be found in \cite[\S2--4]{hain_comp}, \cite[\S3--5]{hain_inf} and \cite[\S3]{hain_hodge_modular}.

Suppose that 
\begin{enumerate}
    \item $\Gamma$ is a discrete group,
    \item $R$ is a reductive group over $\Q$,
    \item $\rho:\Gamma\to R(\Q)$ is a Zariski dense homomorphism.
\end{enumerate}
The \textit{completion of $\Gamma$ relative to $\rho$} (or the \textit{relative completion of $\Gamma$}) is a pro-algebraic group $\cG_{/\Q}$ defined over $\Q$ whose $\Q$-rational points receive a homomorphism $\widetilde{\rho}:\Gamma\to\cG(\Q)$ that lifts $\rho$. The group $\cG$ is an extension 
$$1\to\U\to\cG\to R\to1$$
of $R$ by a pro-unipotent group $\U$, which is the unipotent radical of $\cG$. It has the following universal property: If 
$$1\to U\to G\to R\to 1$$
is an extension of pro-algebraic groups over $\Q$ such that $U$ is pro-unipotent, and that $\rho$ factors through a homomorphism $\phi:\Gamma\to G(\Q)$ via $\Gamma\to G(\Q)\to R(\Q)$, then there exists a unique homomorphism of pro-algebraic groups $\cG\to G$ that commutes with projections to $R$ and that $\phi$ is the composition $\Gamma\to\cG(\Q)\to G(\Q)$.

When $R$ is trivial, $\rho$ is trivial, and then $\cG=\U$ is the unipotent completion of $\Gamma$.

\begin{proposition}[Naturality {\cite[Prop. 3.5]{hain_hodge_modular}}] 
    For $j=1,2$, let $\rho_j:\Gamma_j\to R_j(\Q)$ be Zariski dense homomorphisms from discrete groups to $\Q$-points of reductive groups. Let $\cG_j$ and $\widetilde{\rho}_j:\Gamma_j\to\cG_j(\Q)$ be the completion of $\Gamma_j$ relative to $\rho_j$. If the diagram
    \begin{center}
        \begin{tikzcd}
            \Gamma_1 \ar[r, "\rho_1"] \ar[d, "\phi_\Gamma"'] & R_1(\Q) \ar[d, "\phi_R"]\\
            \Gamma_2 \ar[r, "\rho_2"] & R_2(\Q)
        \end{tikzcd}
    \end{center}
    commutes, then there is a unique homomorphism $\phi_\cG$ such that the following diagram commutes.
    \begin{center}
        \begin{tikzcd}
            \Gamma_1 \ar[r, "\widetilde{\rho}_1"] \ar[d, "\phi_\Gamma"'] \ar[rr, bend left, "\rho_1"] & \cG_1(\Q) \ar[r] \ar[d, "\phi_\cG"] & R_1(\Q) \ar[d, "\phi_R"]\\
            \Gamma_2 \ar[r, "\widetilde{\rho}_2"] \ar[rr, bend right, "\rho_2"'] & \cG_2(\Q) \ar[r] & R_2(\Q)
        \end{tikzcd}
    \end{center}
\end{proposition}

A generalization of Levi's Theorem ensures that the extension for relative completion 
$$1\to\U\to\cG\to R\to1$$
splits and that any two such splittings are conjugate by an element of $\U$. Every splitting induces an isomorphism
$$\cG\cong R\ltimes\U$$
that commutes with the projections to $R$, where the action of $R$ on $\U$ is determined by the splitting. This action is determined by the action of $R$ on the Lie algebra $\u$ of $\U$ as $\U=\exp\u$. To give a presentation of $\cG$, it suffices to give a presentation of $\u$ in the category of $R$-modules. 

By standard arguments, $\u$ has a minimal presentation of the form
$$\u\cong\L(H_1(\u))^\wedge/(\im\varphi)\quad\text{with}\quad\im\varphi\cong H_2(\u)$$
in the category of $R$-modules, where $\L(V)$ denotes the free Lie algebra generated by the vector space $V$, and $\varphi:H_2(\u)\to[\L(H_1(\u))^\wedge,\L(H_1(\u))^\wedge]$ is an injection such that the composite
$$H_2(\u)\to[\L(H_1(\u))^\wedge,\L(H_1(\u))^\wedge]\to\Lambda^2H_1(\u)$$
is the dual to the cup product $\Lambda^2H^1(\u)\to H^2(\u)$.

\begin{theorem}[{\cite[\S3.2]{hain_hodge_modular}}] 
    For every finite dimensional $R$-module, there is a natural homomorphism 
    $$\Hom_R(H_k(\u),V)\cong(H^k(\u)\otimes V)^R\to H^k(\Gamma,V)$$
    that is an isomorphism when $k=1$ and injective when $k=2$. If every irreducible finite dimensional representation of $R$ is absolutely irreducible, then there is a natural $R$-module isomorphism
    $$H^1(\u)\cong\bigoplus_\alpha H^1(\Gamma,V_\alpha)\otimes V_\alpha^*$$
    and a natural $R$-module injection
    $$H^2(\u)\hookrightarrow\bigoplus_\alpha H^1(\Gamma,V_\alpha)\otimes V_\alpha^*$$
    where $\{V_\alpha\}$ is a set of representatives of the isomorphism classes of irreducible finite dimensional $R$-modules, and $V_\alpha^*$ is the dual of $V_\alpha$.
\end{theorem}

Our primary example is when the discrete group $\Gamma$ is the mapping class group $\Gamma_{g,n}$. Note that in this case, its unipotent completion $\Gamma_{g,n}^\un$ is trivial. So it is natural to consider relative completion.
\begin{example}[Relative completion of mapping class groups]\label{eg}
\phantom{}
\begin{enumerate}
    \item 
    Let $\Gamma=\Gamma_{g,n}$, $R=\Sp(H)$, and  $\rho:\Gamma_{g,n}\to\Sp(H)$. Since the image of $\rho$ is $\Sp(H_\Z)$, it is Zariski dense. Denote the completion of $\Gamma_{g,n}$ with respect to $\rho$, i.e. the relative completion of $\Gamma_{g,n}$, by $\cG_{g,n}$ and its pro-unipotent radical by $\U_{g,n}$. Denote the Lie algebras of $\cG_{g,n}$, $\U_{g,n}$ and $\Sp(H)$ by $\g_{g,n}$, $\u_{g,n}$ and $\sp(H)$, respectively. Just like for groups, there is an exact sequence for Lie algebras
    $$0\to\u_{g,n}\to\g_{g,n}\to\sp(H)\to0.$$
    To give a presentation of $\g_{g,n}$, it suffices to give a presentation of $\u_{g,n}$ in the category of $\sp(H)$-modules. The structural presentation of $\u_{g,n}$ will be discussed in \S6.
    \item Hyperelliptic case:
    Let $\Gamma=\Delta_{g,n}$, $R=\Sp(H)$, and  $\rho:\Delta_{g,n}\to\Sp(H)$. Denote the completion of $\Delta_{g,n}$ with respect to $\rho$, i.e. the relative completion of $\Delta_{g,n}$, by $\cD_{g,n}$ and its pro-unipotent radical by $\cV_{g,n}$. Denote the Lie algebras of $\cD_{g,n}$ and $\cV_{g,n}$ by $\d_{g,n}$ and $\v_{g,n}$, respectively. Similarly to the previous case, there is an exact sequence of Lie algebras
    $$0\to\v_{g,n}\to\d_{g,n}\to\sp(H)\to0.$$
    The Lie algebra $\v_{g,n}$ will be discussed in more detail in \S7.
\end{enumerate}
\end{example}

\begin{remark}\label{cd}
    The naturality of relative completion implies that the diagram 
    \begin{center}
    \begin{tikzcd}
        1 \ar[r] & \U_{g,n} \ar[r] & \cG_{g,n} \ar[r] & \Sp(H) \ar[r] \ar[d, equal] & 1 \\
        1 \ar[r] & \cV_{g,n} \ar[r] \ar[u] & \cD_{g,n} \ar[r] \ar[u] & \Sp(H) \ar[r] & 1 \\
    \end{tikzcd}
    \end{center}
    commutes with exact rows. This induces an analogous commutative diagram for Lie algebras as follows.
    \begin{center}
    \begin{tikzcd}
        0 \ar[r] & \u_{g,n} \ar[r] & \g_{g,n} \ar[r] & \sp(H) \ar[r] \ar[d, equal] & 0 \\
        0 \ar[r] & \v_{g,n} \ar[r] \ar[u] & \d_{g,n} \ar[r] \ar[u] & \sp(H) \ar[r] & 0 \\
    \end{tikzcd}
    \end{center}
\end{remark}

\begin{remark}\label{bc}
    Relative completion is compatible with base change. Suppose that $k$ is a field of characteristic 0, so it is a field extension of $\Q$, then $\rho_k:\Gamma\to R(k)$ is Zariski dense in the $k$-group $R_{/k}=R_{/\Q}\otimes_{\Q}k$. When $k=\Q_\ell$, we will often denote $\rho_k$ by $\rho_\ell$. One can define the completion $\cG_{/k}$ of $\Gamma$ relative to $\rho_k$. There is a natural isomorphism
    $$\cG_{/k}\cong\cG_{/\Q}\otimes_{\Q}k$$
    between the $k$-group and the base change of $\Q$-group, see \cite[Cor. 4.14]{hain_comp}, \cite[\S 3.3]{hain_hodge_modular}.  
\end{remark}

\begin{example}
    Let $\Gamma$ be a finite index subgroup of $\Gamma_{g,n}$ containing the Johnson subgroup $K_{g,n}$ \cite{hain_g3}. It is generated by Dehn twists on bounding simple closed curves. The completion of $\Gamma$ relative to the standard representation $\rho$ is isomorphic to $\cG_{g,n}$ by \cite[Cor.~6.7]{hain_g3}. When $n=0$, the Johnson subgroup $K_g$ is the kernel of the Johnson homomorphism $\tau_g:T_g\to \Lambda^3H_\Z/H_\Z$. All the abelian levels $\G_g[m]$, Prym-level mapping class groups, and fundamental groups of moduli spaces of curves with an arbitrary root of their canonical bundle are given as such examples for $\Gamma$ in \cite{hain_g3}.
\end{example}

\subsection{Continuous relative completion of a profinite group} This is analogous to the continuous unipotent completion situation above. Every homomorphism appearing in the construction of relative completion previously is now required to be continuous with respect to the given topology. Suppose that 
\begin{enumerate}
    \item $\widehat\Gamma$ is a profinite group,
    \item $R$ is a reductive group over $\Q_\ell$,
    \item $\rho_\ell:\widehat\Gamma\to R(\Q_\ell)$ is a \textit{continuous} Zariski dense homomorphism.
\end{enumerate}
The \textit{continuous completion of $\widehat\Gamma$ relative to $\rho_\ell$} (or the \textit{continuous relative completion of $\widehat\Gamma$}) is a pro-algebraic group $\cG_{/\Q_\ell}$, defined over $\Q_\ell$, whose $\Q_\ell$-points receive a \textit{continuous} homomorphism $\widetilde{\rho_\ell}:\widehat\Gamma\to\cG(\Q_\ell)$ that lifts $\rho_\ell$. The group $\cG$ is an extension 
$$1\to\U\to\cG\to R\to1$$
of $R$ by a pro-unipotent group $\U$, which is the unipotent radical of $\cG$. It has the following universal property: If 
$$1\to U\to G\to R\to 1$$
is an extension of pro-algebraic groups over $\Q_\ell$ such that $U$ is pro-unipotent, and that $\rho_\ell$ factors through a \textit{continuous} homomorphism $\phi_\ell:\widehat\Gamma\to G(\Q_\ell)$ via $\widehat\Gamma\to G(\Q_\ell)\to R(\Q_\ell)$, then there exists a unique homomorphism of pro-algebraic groups $\cG\to G$ that commutes with projections to $R$ and that $\phi_\ell$ is the composition $\widehat\Gamma\to\cG(\Q_\ell)\to G(\Q_\ell)$.

When $R$ is trivial, $\rho_\ell$ is trivial, then $\cG_{/\Q_\ell}=\U_{/\Q_\ell}$ is the continuous unipotent completion of $\widehat\Gamma$.

Suppose that $\widehat\Gamma$ comes from a discrete group, i.e. it is the profinite completion of a discrete group $\Gamma$. Suppose that $\rho:\Gamma\to R(\Q_\ell)$ is a continuous, Zariski dense homomorphism, so that $\rho_\ell:\widehat\Gamma\to R(\Q_\ell)$ is the continuous extension of $\rho$. The following result is stated in \cite[Thm 6.3]{hain_rational}.

\begin{theorem}\label{cont rel comp iso}
    Suppose $\rho$ and $\rho_\ell$ are given as above. If $\cG_{/\Q_\ell}$ and $\tilde\rho:\Gamma\to\cG(\Q_\ell)$ is the completion of $\Gamma$ relative to $\rho$, then:
    \begin{enumerate}
        \item $\tilde\rho$ is continuous and thus induces a continuous homomorphism $\widetilde{\rho_\ell}:\widehat\Gamma\to\cG(\Q_\ell)$;
        \item $\cG_{/\Q_\ell}$ and $\widetilde{\rho_\ell}$ form the continuous relative completion of $\widehat\Gamma$ with respect to $\rho_\ell$.
    \end{enumerate}
\end{theorem}

\begin{remark}
    This is a generalization of Thm \ref{ctsuc}. In particular, combined with Remark \ref{bc}, the continuous completion of $\widehat{\Gamma_{g,n}[m]}$ with respect to the standard $\rho_\ell:\widehat{\Gamma_{g,n}[m]}\to\Sp(H_{\Q_\ell})$ is $\cG_{g,n}\otimes_{\Q}\Q_\ell$, a $\Q_\ell$-group base changed from a $\Q$-group.
\end{remark}

\section{The Lie algebras $\t_{g,n}$ and $\u_{g,n}$} 

\subsection{Non-trivial central extension}
Recall that the Torelli group $T_{g,n}$ is the kernel of the standard homomorphism $\rho:\Gamma_{g,n}\to\Sp(H)$. The relative completion $\cG_{g,n}$ of $\Gamma_{g,n}$ has unipotent radical $\U_{g,n}$. It is equipped with a homomorphism $\widetilde{\rho}:\Gamma_{g,n}\to\cG_{g,n}(\Q)$ that lifts $\rho$. The restriction of $\widetilde{\rho}$ to $T_{g,n}$ thus has image inside $\U_{g,n}(\Q)$. Since $\U_{g,n}$ is pro-unipotent, it factors through
$$T_{g,n}\to T^\un_{g,n}(\Q)\to\U_{g,n}(\Q).$$
The following result is stated in \cite[Thm. 3.4]{hain_inf}. It is first proved for $g\ge8$ in \cite{hain_comp}, but the argument works for $g\ge3$ in view of \cite[Thm. 3.2]{hain_inf}. The sketch of proof is summarized from \cite[Rmk. on pf. of Thm. 5.7]{hain_johnson}.
\begin{theorem}
    If $g\ge3$, the homomorphism $T^\un_{g,n}\to\U_{g,n}$ is surjective. Its kernel is a copy of the additive group $\bG_a$ contained in the center of $T^\un_{g,n}$. There is a non-trivial central extension 
    $$0\to\Q\to\t_{g,n}\to\u_{g,n}\to0$$
    of pro-nilpotent Lie algebras.
\end{theorem}

\begin{proof}[Proof (Sketch)]
    First, note that the completion of $\Sp(H_\Z)$ relative to its inclusion into $\Sp(H)$ is $\Sp(H)$ when $g\ge3$. Second, the relative completion is right exact. Taking relative completion of each column in the following diagram
    \begin{center}
        \begin{tikzcd}
        1 \ar[r] & T_{g,n} \ar[r] \ar[d] & \Gamma_{g,n} \ar[r] \ar[d] & \Sp(H_\Z) \ar[r] \ar[d] & 1 \\
        1 \ar[r] & 1 \ar[r] & \Sp(H) \ar[r] & \Sp(H) \ar[r] & 1
        \end{tikzcd}
    \end{center}
    we obtain an exact sequence
    $$T^\un_{g,n}\to\cG_{g,n}\to\Sp(H)\to1,$$
    which implies that $T^\un_{g,n}\to\U_{g,n}$ is surjective. The kernel of this map is central and there is an exact sequence
    $$H_2(\Sp(H_\Z);\Q)\to T^\un_{g,n}\to\U_{g,n}\to1.$$
    The injectivity on the left is equivalent to non-triviality of the biextension line bundle over $\M_g$ associated to the Ceresa cycle. This follows from a computation of Morita \cite[5.8]{morita}. By taking Lie algebras, we obtain the non-trivial central extension.
\end{proof}

\subsection{Hodge theory
}
Hodge theory and representation theory will be used to further study the Lie algebras $\t_{g,n}$ and $\u_{g,n}$. They are naturally equipped with $\Q$-mixed Hodge structures (MHS). Weight filtration is a part of the data of a MHS. By taking the associated weight graded functor on MHS, there is no loss of information because the weight filtration of MHS satisfies good exactness property. This is a powerful tool for studying structural presentations of the Lie algebras $\t_{g,n}$ and $\u_{g,n}$. Since the weight graded pieces are naturally $\Sp(H)$-modules, representation theory will be reviewed next in \S\ref{Sp reps}.

Classically, Deligne puts functorial MHS on cohomology groups of complex algebraic varieties \cite{deligne}. Later, Hain puts MHS on fundamental groups of varieties which are dependent on base points \cite{hain}. In fact, this means that MHS can be put on (coordinate rings of) unipotent completions of fundamental groups and their Lie algebras. We start with the following standard result.
\begin{theorem}
    For each choice of a base point $x$ on an algebraic curve $C$, there is a canonical MHS on the Lie algebra $\p$ of the unipotent completion $\cP$ of the fundamental group $\pi_1(C,x)$.
\end{theorem}

More generally, MHS can be put on relative completions of fundamental groups \cite{hain_hodge}. For each choice of a base point $x$ in a smooth quasi-projective complex variety $X$, denote by $\pi_1(X,x)$ the topological fundamental group. Suppose that $\V$ is a polarized variation of $\Q$-Hodge structure over $X$. Denote its fiber over $x$ by $V_x$. The Zariski closure of the image of the monodromy representation
$$\rho_x:\pi_1(X,x)\to\Aut(V_x)$$
is reductive over $\Q$ \cite{deligne}. Denote this reductive group by $R$, and we obtain a Zariski dense homomorphism $\rho:\pi_1(X,x)\to R(\Q)$.\footnote{We suppress the base point $x$ subscripts for convenience, but keep its dependence in mind.} Denote the completion of $\pi_1(X,x)$ relative to $\rho:\pi_1(X,x)\to R(\Q)$ by $\cG$, and its unipotent radical by $\U$. Denote the Lie algebras of $\cG$, $\U$, and $R$ by $\g$, $\u$, and $\r$, respectively.

\begin{theorem}[{\cite[Thm.~13.1, Cor.~13.2]{hain_hodge}}]
    For each choice of a base point $x\in X$, there are canonical MHS on the coordinate rings $\O(\cG)$ and $\O(\U)$ for which the product, coproduct and antipode are morphisms of MHS. Consequently, there are canonical MHS on the Lie algebras $\g$ and $\u$ for which the Lie brackets are morphisms of MHS. Moreover, the weight filtrations on these objects satisfy that
    $$W_{-1}\O(\cG)=0,\quad W_0\O(\cG)=\O(R)$$
    and 
    $$W_0\g=\g,\quad W_{-1}\g=\u,\quad\Gr^W_0\g\cong\r.$$
\end{theorem}

\begin{remark}
    This generalizes naturally to the case when $X$ is an orbifold with the orbifold fundamental group $\pi^\orb_1(X,x)$.
\end{remark}

\begin{example}[MHS on relative completions] Apply the above theorem and remark to Example \ref{eg}.
    \begin{enumerate}
        \item For each choice of a base point $x\in\M_{g,n}^\an$, there is a natural isomorphism of $\Gamma_{g,n}$ with the orbifold fundamental group $\pi_1^\orb(\M^\an_{g,n},x)$. Once the base point $x$ is chosen, this isomorphism is fixed. We identify $\Gamma_{g,n}$ with $\pi_1^\orb(\M^\an_{g,n},x)$, and take its relative completion to obtain $\cG_{g,n}$. There are canonical MHS on its coordinate ring $\O(\cG_{g,n})$, and on the Lie algebras $\g_{g,n}$ and $\u_{g,n}$ with
        $$W_0\g_{g,n}=\g_{g,n},\quad W_{-1}\g_{g,n}=\u_{g,n},\quad\Gr^W_0\g_{g,n}\cong\sp(H)$$
        where $\sp(H)$ denotes the Lie algebra of $\Sp(H)$.
        \item For each choice of a base point $x\in\H^\an_{g,n}$, there is a natural isomorphism of $\Delta_{g,n}$ with the orbifold fundamental group $\pi_1^\orb(\H^\an_{g,n},x)$. Then we identify $\Delta_{g,n}$ with $\pi_1^\orb(\H^\an_{g,n},x)$, and take its relative completion $\cD_{g,n}$. There are canonical MHS on the coordinate ring $\O(\cD_{g,n})$, and on the Lie algebras $\d_{g,n}$ and $\v_{g,n}$ with 
        $$W_0\d_{g,n}=\d_{g,n},\quad W_{-1}\d_{g,n}=\v_{g,n},\quad\Gr^W_0\d_{g,n}\cong\sp(H).$$
    \end{enumerate}
\end{example}

By naturality of relative completion and functoriality of MHS, it is conceivable that natural maps between these Lie algebras with MHS, such as those appeared in Remark \ref{cd}, are morphisms of MHS.

\begin{corollary} \label{v to u mhs map}
    For each choice of a base point $x\in\H^\an_{g,n}$, the natural map $d\tilde{i}_*:\v_{g,n}\to\u_{g,n}$ is a morphism of MHS.
\end{corollary}
\begin{proof}
    We only need to show that this map is naturally defined.
    Once a base point $x\in\H^\an_{g,n}$ is chosen, the closed immersion $i:\H^\an_{g,n}\to\M^\an_{g,n}$ induces a homomorphism of fundamental groups
    $i_*:\pi^\orb_1(\H^\an_{g,n},x)\to\pi^\orb_1(\M^\an_{g,n},x).$
    Taking their relative completions with respect to $\Sp(H)$, we obtain $\tilde{i}_*:\cD_{g,n}\to\cG_{g,n}$ and, by restricting to the unipotent radicals,
    $$\tilde{i}_*:\cV_{g,n}\to\U_{g,n}.$$
    Finally, taking Lie algebras, we obtain the natural map
    $$d\tilde{i}_*:\v_{g,n}\to\u_{g,n}.$$
\end{proof}

The existence of a canonical MHS on $\t_{g,n}$ is a bit different. Nonetheless, it can be naturally lifted from the MHS on $\u_{g,n}$.
\begin{theorem}[{\cite[Thm.~4.10]{hain_inf}}]\label{t to u mor of mhs}
    For each choice of a base point $x\in\M^\an_{g,n}$, there is a canonical MHS on $\t_{g,n}$ for which the Lie bracket and the natural map $\t_{g,n}\to\u_{g,n}$ are morphisms of MHS.
\end{theorem}

\subsection{\textcolor{black}{$\Sp(H)$ representations}}\label{Sp reps} 
Now we review the representation theory of $\Sp(H)$.

Denote the standard representation by $H$. Denote the Lie algebra of $\Sp(H)$ by $\sp(H)$. Every irreducible representation of $\Sp(H)$ is a subrepresentation of some tensor power $H^{\otimes m}$ of $H$. Fix a symplectic basis $a_1,\cdots, a_g, b_1, \cdots, b_g$ of $H$. Set
$$\theta:=a_1\wedge b_1+\cdots+a_g\wedge b_g\in\Lambda^2 H.$$
With respect to this chosen basis, denote by $\h$ the torus in $\sp(H)$ consisting of diagonal matrices with coordinates $t=(t_1,\cdots, t_g)$ with
$$t\cdot a_i=t_ia_i\quad\text{and}\quad t\cdot b_i=-t_ib_i.$$ 
This is a Cartan subalgebra. In turn, a Borel subalgebra can be chosen so that a fundamental set of weights $\{\lambda_j:1\le j\le g\}$ can be defined by
$$\lambda_j(t)=t_1+\cdots+t_j.$$
Denote by $V(\lambda)$ the irreducible representation having a highest weight $\lambda$ that corresponds to a positive integral linear combination $\sum_{j=1}^gn_j\lambda_j$. In particular when $k\ge2$, the $k$-th fundamental representation is
$$V(\lambda_k)=\Lambda^kH/(\theta\wedge\Lambda^{k-2}H)=:\Lambda^k_0H.$$

Computing the weight graded pieces of the Lie algebras $\p$, $\t_{g,n}$ and $\u_{g,n}$, as it turns out, is essentially equivalent to computing the corresponding graded pieces for the lower central series of these Lie algebras. 

\begin{lemma}[{\cite[Lem.~4.7]{hain_inf}}]\label{W=L}
    Suppose that $\g$ is a pronilpotent Lie algebra in the category of mixed Hodge structures with finite dimensional $H_1$. If the induced MHS on $H_1(\g)$ is pure of weight $-1$, then $W_{-l}\g$ is the $l$-th term $L^l\g$ of the lower central series of $\g$.
\end{lemma}

\begin{proposition}[{\cite[Prop.~8.4]{hain_inf}}]\label{sp str of p}
 The weight filtration of $\p$ satisfies that
 $$\p=W_{-1}\p$$
 and the $l$-th term of the lower central series of $\p$ is
 $$L^l\p=W_{-l}\p.$$ 
 Suppose that $g\geq 3$. There is a natural $\Sp(H)$-equivariant Lie algebra isomorphism
    $$\Gr^W_\bullet\p\cong\L(H)/(\theta).$$
    The highest weight decomposition of the first few weight graded pieces of $\p$ are:
    $$\Gr^W_{-m}\p\cong\begin{cases}
        V(\lambda_1) & m=1; \\
        V(\lambda_2) & m=2; \\
        V(\lambda_1+\lambda_2) & m=3; \\
        V(2\lambda_1+\lambda_2)\oplus V(2\lambda_1)\oplus V(\lambda_1+\lambda_3) & m=4. 
    \end{cases}$$
\end{proposition}

Explicit presentations for derivation algebras of $\p$ will be useful, so we record their computations here. 

The adjoint action of elements of $\p$ can be viewed as inner derivations. Denote the set of all inner derivations by $\Inn\Der\p$. It is a Lie subalgebra inside the derivation Lie algebra $\Der\p$. Since $\p$ has trivial center, 
$$\p\cong\Inn\Der\p,$$ 
the adjoint action embeds $\p$ itself into $\Der\p$, and the quotient is the outer derivations 
$$\Out\Der\p:=\Der\p/\Inn\Der\p$$
of $\p$. This gives the exact sequence
$$0\to\p\to\Der\p\to\Out\Der\p\to0.$$

Note that these derivation algebras admit weight filtrations induced from that of $\p$. As usual, we shall work with the associated weight graded sequence
$$0\to\Gr^W_\bullet\p\to\Gr^W_\bullet\Der\p\to\Gr^W_\bullet\Out\Der\p\to0.$$

\begin{corollary}[{\cite[Cor.~9.2, Cor.~9.4]{hain_inf}}]
    For $g\geq 3$, we have
    $$\Gr^W_{-m}\Der\p\cong\begin{cases}
        V(\lambda_3)\oplus V(\lambda_1) & m=1;\\
        V(2\lambda_2)\oplus V(\lambda_2) & m=2;\\
        V(2\lambda_1+\lambda_3)\oplus V(\lambda_1+\lambda_2)\oplus V(3\lambda_1) & m=3.
    \end{cases}$$
    And
    $$\Gr^W_{-m}\Out\Der\p\cong\begin{cases}
        V(\lambda_3) & m=1;\\
        V(2\lambda_2) & m=2;\\
        V(2\lambda_1+\lambda_3)\oplus V(3\lambda_1) & m=3.
    \end{cases}$$
\end{corollary}

For convenience in later sections, we shall also denote 
$$\Gr^W_{-m}\Der\p,\,\,\,\, \Gr^W_{-m}\Inn\Der\p,\,\,\,\,\text{ and }\,\,\,\, \Gr^W_{-m}\Out\Der\p$$ by $$\Der_{-m}\p,\,\,\,\,\text{ and }\,\,\,\, \Inn\Der_{-m}\p,\,\,\,\, \Out\Der_{-m}\p$$
respectively.

\subsection{Presentations for $\Gr^W_\bullet\t_{g,n}$ and $\Gr^W_\bullet\u_{g,n}$}

Johnson's famous result on the abelianization of the Torelli groups implies that for $g\ge 3$, there are \textcolor{black}{$\Sp(H)$}-equivariant isomorphisms
$$H_1(T_{g,n};\Q)=H_1(\t_{g,n};\Q)\cong H_1(\u_{g,n};\Q)\cong\Lambda_0^3H\oplus H^{\oplus n}$$
where $\Lambda_0^3H$ is the irreducible $\Sp(H)$-module corresponding to the third fundamental weight $\lambda_3$. In particular, when $n=1$,
$$H_1(\t_{g,1})\cong H_1(\u_{g,1})\cong\Lambda^3_0H\oplus H\cong\Lambda^3H\cong\Der_{-1}\p.$$
Hodge theory implies that $H_1(\u_{g,n})\cong H_1(\t_{g,n})$ is pure of weight $-1$, since $H$ is pure of weight $-1$ and so is $\Lambda_0^3 H$, which is the primitive third degree homology of the jacobian with a Tate twist. By Lemma \ref{W=L}, the weight filtration of $\u_{g,n}$ (or $\t_{g,n}$) is essentially its lower central series. 

\begin{proposition}[{\cite[Cor.~4.8]{hain_inf}}]\label{weight agrees lcs}
    The weight filtration on $\u_{g,n}$ satisfies that
    $$\u_{g,n}=W_{-1}\u_{g,n}$$
    and the $l$-th term of the lower central series of $\u_{g,n}$ is
    $$L^l\u_{g,n}=W_{-l}\u_{g,n}.$$
    Similarly, the weight filtration on $\t_{g,n}$ satisfies that
    $$\t_{g,n}=W_{-1}\t_{g,n}$$
    and the $l$-th term of the lower central series of $\t_{g,n}$ is
    $$L^l\t_{g,n}=W_{-l}\t_{g,n}.$$
\end{proposition}

In turn, explicit presentations can be computed for $\Gr^W_\bullet\u_{g,n}$. The generator $H_1(\u_{g,n})$ is known as above, so we only need to determine relations. The strategy is to find a potentially large homomorphism image of $\u_{g,n}$ (or its graded variant). The relations on the image provide an upper bound on relations of the preimage. A lower bound for relations is found by computing brackets of carefully chosen elements in the Torelli group. In case the upper bound agrees with the lower bound, all relations are serendipitously found.

The sought after homomorphism comes from the Lie algebra version (or its associated weight graded variant)
$$\u_{g,1}\to\Der\p\quad\text{or}\quad\Gr^W_\bullet\u_{g,1}\to\Gr^W_\bullet\Der\p$$
of the monodromy representation
$$\rho: \Gamma_{g,1}\cong\pi_1(\M^\an_{g,1}, *)\to\Aut\p$$
of the universal curve $\pi: \cC^\an_{g,1}\to\M^\an_{g,1}$. Here $*=[C,x]$ is a moduli point in $\M^\an_{g,1}$, and $\p$ is the Malcev Lie algebra of $\pi_1(C,x)$ as before.   

\begin{theorem}[\cite{hain_inf}, \cite{hain_g3}, {\cite[Thm.~8.2]{hain_johnson}}]
    For all $g\ge 4$, the graded Lie algebra $\Gr^W_\bullet\u_{g,n}$ is quadratically presented. In particular,
    $$\Gr^W_\bullet\u_{g,1}\cong\L(\Lambda^3 H)/(R_2)$$
    where $R_2$ is the kernel of the Lie bracket    $$\Lambda^2\Der_{-1}\p\to\Der_{-2}\p.$$
    In genus $3$, the graded Lie algebra $\u_{3,n}$ has non-trivial quadratic and cubic relations. These are determined by the condition that $\Gr^W_\bullet\u_{3,1}\to\Gr^W_\bullet\Der\p$ is an isomorphism in weights $-1$, $-2$ and injective in weight $-3$.
\end{theorem}

\begin{remark}
As for $\Gr^W_\bullet\t_{g,n}$, it is a central extension
    $$0\to\Q\to\Gr^W_\bullet\t_{g,n}\to\Gr^W_\bullet\u_{g,n}\to0.$$
The central $\Q$ lies in the image of the map
$$\Lambda^2H_1(\t_g)\to\Gr^W_{-2}\t_g$$
induced by the Lie bracket. It has to be of weight $-2$, as the Lie bracket is a morphism of MHS. Since it is one dimensional, it must be $\Q(1)$.
\end{remark}

\begin{proposition}\label{t to u sp}
    The map $\Gr^W_\bullet\t_{g,n}\to\Gr^W_\bullet\u_{g,n}$ is \textcolor{black}{$\Sp(H)$}-equivariant.
\end{proposition}
\begin{proof}
    This map comes from $\t_{g,n}\to\u_{g,n}$, which is $\Gamma_{g,n}$-equivariant, cf. \cite[Pf. of Thm.~4.10]{hain_inf}. The $\Gamma_{g,n}$ action on $\t_{g,n}$ is induced from the action of $\Gamma_{g,n}$ on $T_{g,n}$ by conjugation. Taking associated weight graded for $\t_{g,n}\to\u_{g,n}$, the map $\Gr^W_\bullet\t_{g,n}\to\Gr^W_\bullet\u_{g,n}$ is still $\Gamma_{g,n}$-equivariant. But $T_{g,n}$ acts trivially on each weight graded piece, which by Lemma \ref{W=L} are the same as the graded pieces of lower central series of $\t_{g,n}$, so the map is $\Sp(H_\Z)$-equivariant. \textcolor{black}{Since the action of $\Sp(H_\Z)$ on each graded piece extends to a rational representation of $\Sp(H)$, the result follows.}
\end{proof}

\section{The Lie algebras $\d_{g,n}$ and $\v_{g,n}$}
Consider the monodromy representation $\rho^\hyp_\ell:\pi_1^\alg(\H_{g,n/\C})\to \Sp(H_\Ql)$ and the projection $\pi^\hyp_\ast:\pi_1^\alg(\cC_{\H_{g,n/\C}})\to \pi_1^\alg(\H_{g,n/\C})$ induced by $\pi^\hyp$.
Let 
$\rho^\hyp_{\cC,\ell}: \pi_1^\alg(\cC_{\H_{g,n/\C}})\to \Sp(H_\Ql)$ be the representation $\rho^\hyp_\ell\circ \pi^\hyp_\ast$.
Denote the continuous relative completion of $\pi_1^\alg(\H_{g,n/\C})$ and $\pi_1^\alg(\cC_{\H_{g,n/\C}})$ with respect to  $\rho^\hyp_\ell$ and $\rho^\hyp_{\cC,\ell}$ by $\cD_{g,n/\Ql}$ and $\cD_{\cC_{g,n}/\Ql}$ and their unipotent radicals by $\cV_{g,n/\Ql}$ and $\cV_{\cC_{g,n}/\Ql}$, respectively. Let their Lie algebras be denoted by $\d_{g,n/\Ql}$, $\d_{\cC_{g,n}/\Ql}$, $\v_{g,n/\Ql}$, and $\v_{\cC_{g,n}/\Ql}$, respectively. \\
\indent 
Recall that the Lie algebra $\d_{g,n}$ is equipped with a MHS and admits natural weight filtration $W_\bullet\d_{g,n}$ defined over $\Q$. 
Since $\cD_{g,n/\Ql}\cong \cD_{g,n}\otimes \Ql$ and so $\d_{g,n/\Ql}\cong \d_{g,n}\otimes_\Q\Ql$, the weight filtration of $\d_{g,n}$ lifts to that of $\d_{g,n/\Ql}$, and it is preserved by the adjoint action of $\cD_{g,n/\Ql}$. Similarly, $\d_{\cC_{g,n}/\Ql}$ admits a weight filtration $W_\bullet\d_{\cC_{g,n}/\Ql}$ as well. 

Associated to the universal hyperelliptic curve $\pi^\hyp: \cC_{\H_{g,n/\C}} \to \H_{g,n/\C}$, there exists an exact sequence of algebraic fundamental groups:
$$
1 \to \pi_1^\alg(C) \to \pi_1^\alg(\cC_{\H_{g,n/\C}}) \overset{\pi^\hyp_\ast}\to \pi_1^\alg(\H_{g,n/\C}) \to 1.
$$

Applying the continuous relative completion to this sequence, we obtain the exact sequence of pro-algebraic groups over $\Ql$:
$$
1 \to\cP_{\Ql}\to \cD_{\cC_{g,n}/\Ql}\to \cD_{g,n/\Ql}\to 1.
$$
Restricting to the unipotent radicals of the middle and right terms yields the exact sequence of pro-unipotent groups over $\Ql$:
\begin{equation*}
1 \to\cP_{\Ql}\to \cV_{\cC_{g,n}/\Ql}\to \cV_{g,n/\Ql}\to 1,
\end{equation*}
 and hence the exact sequence of pro-nilpotent Lie algebras:
\begin{equation}\label{hyp_seq}
0 \to \p_\Ql \to \v_{\cC_{g,n}/\Ql} \to \v_{g,n/\Ql} \to 0.
\end{equation}
For the case $n = 0$, there is an isomorphism $\cC_{\H_{g/\C}} \cong \H_{g,1/\C}$, and thus $\pi_1^\alg(\cC_{\H_{g/\C}})$ is isomorphic to the profinite completion $\widehat{\Delta_{g,1}}$. Consequently, the sequence \eqref{hyp_seq} yields to:
\begin{equation}\label{n=1 hyp}
0 \to \p_\Ql \to \v_{g,1/\Ql} \to \v_{g/\Ql} \to 0.
\end{equation}
Next, taking the quotient of each term in the sequence \eqref{n=1 hyp} by $W_{-5}$ and applying the exact functor $\Gr^W_\bullet$ to the resulting sequence, we obtain the following exact sequence of 4-step graded Lie algebras:
\[
0 \to \Gr^W_\bullet (\p_\Ql / W_{-5}) \to \Gr^W_\bullet (\v_{g,1/\Ql} / W_{-5}) \to \Gr^W_\bullet (\v_{g/\Ql} / W_{-5}) \to 0.
\]
Denote $\Gr^W_\bullet (\v_{g,n/\Ql} / W_{-5})$
by $\h_{g,n}$ and the projection $\h_{g,1} \to \h_{g}$ by $\beta_{g}$. 
\begin{theorem}[{\cite[Thm.~8.4, Cor.~8.6]{wat_hyp_univ}\cite[Thm.~9.4]{wat_rk_hyp_univ}}]\label{key result from hyp}
With the notation as above, if $g \geq 3$, the projection $\beta_g$ admits no $\mathrm{Sp}(H_{\mathbb{Q}_\ell})$-equivariant graded Lie algebra section. Consequently, the projection 
\[
\mathrm{Gr}^W_\bullet \mathfrak{v}_{g,1/\mathbb{Q}_\ell} \to \mathrm{Gr}^W_\bullet \mathfrak{v}_{g/\mathbb{Q}_\ell}
\]
also admits no such section.
\end{theorem}
In \cite{wat_hyp_univ}, the proof of the theorem used the weighted completion (see \cite[\S 7]{wat_hyp_univ}) of $\pi_1^\alg(\H_{g,n/k})$, where the field $k$  is a field of characteristic zero satisfying a mild condition. In this paper, in order to use an analogous result, we need to replace the weighted completion with the (continuous) relative completion of $\pi_1^\alg(\H_{g,n/\C})$. A key point that allows us to use similar computations done in \cite{wat_hyp_univ} is that $H_1(\v_{g})$ is pure of weight $-2$ using the weight filtration obtained from the Hodge theory. The detail of this fact can be found in \cite[\S 6]{wat_rk_hyp_univ}. The following is a sketch of the proof of the theorem for the case over $\Q$.


By \cite[Thm.~6.14]{wat_hyp_univ}, there exist two derivations $d^\out_1$ and $d^\out_2$ in $V(2\lambda_2)\subset \Der_{-2}\p$ such that the bracket $[d^\out_1, d^\out_2]$ lies in $\Inn\Der_{-4}\p$. 

Consider the exact sequence:
$$
0\to \Gr^W_\bullet\p\to\Gr^W_\bullet\v_{g,1}\to \Gr^W_\bullet\v_g\to 0.
$$
Via the adjoint action,
there exists an $\Sp(H)$-equivariant graded Lie algebra homomorphism $\Gr^W_\bullet\v_{g,1}\to \Gr^W_\bullet\Der\p$. This yields an $\Sp(H)$-equivariant map $\Gr^W_{-2}\v_{g,1}\to \Der_{-2}\p$. The fact that $H_1(\v_g)$ is pure of weight $-2$ implies that there exist $\Sp(H)$-equivariant isomorphisms
$$
\Gr^W_{-2}\v_{g,1} \cong \Gr^W_{-2}\p \oplus \Gr^W_{-2}\v_g\cong \Gr^W_{-2}\p \oplus H_1(\v_g).
$$
The analogous results of \cite[Lemma~7.9, Prop.~8.3] {wat_hyp_univ} for our case provide two vectors $\delta_1$ and $\delta_2$ in $H_1(\v_g)\subset\Gr^W_{-2}\v_{g,1}$ such that  their images under the adjoint map $\Gr^W_{-2}\v_{g,1}\to\Der_{-2}\p$
are $d^\out_1$ and $d^\out_2$, respectively, and the bracket $[\delta_1, \delta_2]$ is a nontrivial vector in $\Gr^W_{-4}\p$. Denote the images of $\delta_1$ and $\delta_2$ in $\Gr^W_{-2}\v_g$ by $\bar\delta_1$ and $\bar\delta_2$, respectively. Suppose that $\beta_g$ admits an $\Sp(H)$-equivariant graded Lie algebra section $s$. Since the bracket $[\delta_1, \delta_2]$ is in $\Gr^W_{-4}\p$, it follows that  $[\bar\delta_1,\bar\delta_2]=0$ in $\Gr^W_{-4}\v_g$. By an analogous result of \cite[Lemma 7.9]{wat_hyp_univ}, $s$ necessarily maps $\bar \delta_1$ and $\bar\delta_2$ to $\delta_1$ and $\delta_2$, respectively. This is due to the fact that there is no $\Sp(H)$-equivariant map from $H_1(\v_g)$ to $\Gr^W_{-2}\p\cong V(\lambda_2)$. 
Since $s$ is a Lie algebra homomorphism, we would have $s(0)= [\delta_1, \delta_2]\not=0$, a contradiction. 

\section{Proofs of the main results}
Denote the continuous relative completion of $\pi_1^\alg(\M_{g,n/\C})$  by $\cG_{g,n/\Ql}$ and denote its unipotent radical by $\U_{g,n/\Ql}$. Denote the Lie algebras of $\cG_{g,n/\Ql}$ and $\U_{g,n/\Ql}$ by $\g_{g,n/\Ql}$ and $\u_{g,n/\Ql}$, respectively. By Remark \ref{bc}, there is a natural isomorphism 
$$
\cG_{g,n/\Ql}\cong \cG_{g,n}\otimes \Ql.
$$

\begin{proposition}\label{sp equiv} Let $g\geq 2$ and $n\geq 0$. 
The closed immersion $i:\H_{g,n/\C}\hookrightarrow\M_{g,n/\C}$ induces an $\Sp(H_\Ql)$-equivariant graded Lie algebra map $\Gr^W_\bullet\v_{g,n/\Ql}\to \Gr^W_\bullet\u_{g,n/\Ql}$.
\end{proposition}
\begin{proof}

By Corollary \ref{v to u mhs map}, the natural Lie algebra map 
$d\tilde{i}_\ast: \v_{g,n}\to \u_{g,n}$ induced by $i$ is a morphism of MHSs. 
Since $\cD_{g,n/\Ql} \cong \cD_{g,n} \otimes_{\mathbb{Q}} \Ql$ and $\cG_{g,n/\Ql} \cong \cG_{g,n} \otimes_{\mathbb{Q}} \Ql$, it follows that the weight filtrations on $\v_{g,n}$ and $\u_{g,n}$ induce weight filtrations on $\v_{g,n/\Ql}$ and $\u_{g,n/\Ql}$, respectively, and that $d\tilde{i}_\ast$ induces a $\cD_{g,n/\Ql}$-equivariant Lie algebra map $d\tilde{i}_{\ast/\Ql}: \v_{g,n/\Ql} \to \u_{g,n/\Ql}$ that preserves the weight filtrations. Consequently, this induces a graded Lie algebra homomorphism 
\[
\Gr^W_\bullet d\tilde{i}_{\ast/\Ql}: \Gr^W_\bullet \v_{g,n/\Ql} \to \Gr^W_\bullet \u_{g,n/\Ql}.
\]
Since $\v_{g,n/\Ql} = W_{-1} \v_{g,n/\Ql}$, it follows that the adjoint action of $\v_{g,n/\Ql}$ is trivial on each graded piece $\Gr^W_m \v_{g,n/\Ql}$ and trivial on $\Gr^W_m \u_{g,n/\Ql}$ via $d\tilde{i}_{\ast/\Ql}$. This implies that the adjoint action of $\cV_{g,n/\Ql}$ on the associated graded Lie algebras is the identity. Hence the action of $\cD_{g,n/\Ql}$ factors through $\Sp(H_{\Ql})$. Therefore, the map $\Gr^W_\bullet d\tilde{i}_{\ast/\Ql}$ is $\Sp(H_{\Ql})$-equivariant.

\end{proof}

\begin{corollary}\label{section u to v}Let $g\geq 2$ and $n\geq 0$.
Each $\Sp(H_\Ql)$-equivariant graded Lie algebra section of $\Gr^W_\bullet\u_{\cC_{g,n}/\Ql}\to \Gr^W_\bullet\u_{g,n/\Ql}$ induces an $\Sp(H_\Ql)$-equivariant section of $\Gr^W_\bullet\v_{\cC_{g,n}/\Ql}\to \Gr^W_\bullet\v_{g,n/\Ql}$.
\end{corollary}
\begin{proof}
Consider the closed immersion $i:\H_{g,n/\C}\hookrightarrow\M_{g,n/\C}$ and the commutative diagram:
$$
\xymatrix@R=1em@C=2em{
\cC_{\H_{g,n/\C}}\ar[r]\ar[d]_{\pi^\hyp}&\cC_{g,n/\C}\ar[d]^{\pi}\\
\H_{g,n/\C}\ar[r]^i&\M_{g,n/\C}
}
$$
This induces the pullback diagram of algebraic fundamental groups:
$$
\xymatrix@R=1em@C=2em{
1\ar[r]&\pi_1^\alg(C)\ar[r]\ar@{=}[d]&\pi_1^\alg(\cC_{\H_{g,n/\C}})\ar[r]\ar[d]&\pi_1^\alg(\H_{g,n/\C})\ar[r]\ar[d]_{i_\ast}&1\\
1\ar[r]&\pi_1^\alg(C)\ar[r]&\pi_1^\alg(\cC_{g,n/\C})\ar[r]&\pi_1^\alg(\M_{g,n/\C})\ar[r]&1.
}
$$
Taking the continuous unipotent completion of $\pi_1^\alg(C)$, the relative completion of middle and right terms in the diagram with respect to their monodromy representations to $\Sp(H_\Ql)$, we obtain the commutative diagram of pro-algebraic groups over $\Ql$:
$$
\xymatrix@R=1em@C=2em{
1\ar[r]&\cP_\Ql\ar[r]\ar@{=}[d]&\cD_{\cC_{g,n}/\Ql}\ar[r]\ar[d]&\cD_{g,n/\Ql}\ar[r]\ar[d]_{\tilde{i}_{\ast/\Ql}}&1\\
0\ar[r]&\cP_\Ql\ar[r]&\cG_{\cC_{g,n}/\Ql}\ar[r]&\cG_{g,n/\Ql}\ar[r]&1,
}
$$
where the rows are exact and the morphism $\tilde{i}_{\ast/\Ql}$ is the one induced by the naturality of continuous relative completion. Restricting to the unipotent radicals of the middle and right terms, we obtain the diagram of pro-unipotent groups with exact rows:
$$
\xymatrix@R=1em@C=2em{
1\ar[r]&\cP_\Ql\ar[r]\ar@{=}[d]&\cV_{\cC_{g,n}/\Ql}\ar[r]\ar[d]&\cV_{g,n/\Ql}\ar[r]\ar[d]_{\tilde{i}_{\ast/\Ql}}&1\\
0\ar[r]&\cP_\Ql\ar[r]&\U_{\cC_{g,n}/\Ql}\ar[r]&\U_{g,n/\Ql}\ar[r]&1.
}
$$
Now, taking the Lie algebras of the completions, we obtain the commutative diagram of pro-nilpotent Lie algebras:
$$
\xymatrix@R=1em@C=2em{
0\ar[r]&\p_\Ql\ar[r]\ar@{=}[d]&\v_{\cC_{g,n}/\Ql}\ar[r]\ar[d]&\v_{g,n/\Ql}\ar[r]\ar[d]_{d\tilde{i}_{\ast/\Ql}}&0\\
0\ar[r]&\p_\Ql\ar[r]&\u_{\cC_{g,n}/\Ql}\ar[r]&\u_{g,n/\Ql}\ar[r]&0,
}
$$
where the map $d\tilde{i}_{\ast/\Ql}$ is the Lie algebra homomorphism induced by the morphism $\tilde{i}_{\ast/\Ql}$.
Since the functor $\Gr^W_\bullet$ is exact, we have the pullback diagram of graded Lie algebras:
$$
\xymatrix@R=1em@C=2em{
0\ar[r]&\Gr^W_\bullet\p_\Ql\ar[r]\ar@{=}[d]&\Gr^W_\bullet\v_{\cC_{g,n}/\Ql}\ar[r]\ar[d]&\Gr^W_\bullet\v_{g,n/\Ql}\ar[r]\ar[d]_{\Gr^W_\bullet d\tilde{i}_{\ast/\Ql}}&0\\
0\ar[r]&\Gr^W_\bullet\p_\Ql\ar[r]&\Gr^W_\bullet\u_{\cC_{g,n}/\Ql}\ar[r]&\Gr^W_\bullet\u_{g,n/\Ql}\ar[r]&0,
}
$$
where the rows are exact. By Proposition \ref{sp equiv}, the homomorphism $\Gr^W_\bullet d\tilde{i}_{\ast/\Ql}$ is $\Sp(H_\Ql)$-equivariant. It then follows that each $\Sp(H_\Ql)$-equivariant Lie algebra section of $\Gr^W_\bullet\u_{\cC_{g,n}/\Ql}\to \Gr^W_\bullet\u_{g,n/\Ql}$ induces an $\Sp(H_\Ql)$-equivariant section of $\Gr^W_\bullet\v_{\cC_{g,n}/\Ql}\to \Gr^W_\bullet\v_{g,n/\Ql}$.
\end{proof}
As an immediate consequence of Theorem \ref{key result from hyp} and Corollary \ref{section u to v}, we have the following result.
\begin{theorem} \label{no sec for u} With the notation as above, if $g\geq 3$, then $\Gr^W_\bullet\u_{g,1/\Ql}\to \Gr^W_\bullet\u_{g/\Ql}$ admits no $\Sp(H_\Ql)$-equivariant graded Lie algebra section. 
\end{theorem}
\subsection{Proof of Theorem \ref{main thm for subgrp of mg}}
Suppose that $g\geq 3$ and $\G$ is a finite-index subgroup of $\G_g$ containing $K_g$. Let $\G'$ be the pre-image of $\G$ under the projection $\G_{g,1}\to \G_g$. There exists an exact sequence
$$
1\to \pi_1(S_g)\to \G'\to \G\to 1,
$$
and it is the pullback of the Birman exact sequence for $\G_g$.
We may identify $\G$ and $\G'$ as the orbifold fundamental groups of the underlying analytic orbifolds of finite covers $\M_\G$ and $\M_{\G'}$ of $\M_{g/\C}$ and $\M_{g,1/\C}$, respectively. Denote the relative completions of $\G$ and $\G'$ by $\cG_\G$ and $\cG_{\G'}$ and their unipotent radicals by $\U_\G$ and $\U_{\G'}$, respectively. The Lie algebras of $\cG_\G$, $\cG_{\G'}$, $\U_\G$, and $\U_{\G'}$ are denoted by $\g_\G$, $\g_{\G'}$, $\u_\G$, and $\u_{\G'}$, respectively. By the naturality of relative completion, the projections $\M_\G\to \M_{g/\C}$ and $\M_{\G'}\to \M_{g,1/\C}$ induce morphisms of MHSs $\g_\G\to \g_g$ and $\g_{\G'}\to \g_{g,1}$, which are isomorphisms by \cite[Cor.~6.7]{hain_g3}. Hence, we have  $\Sp(H)$-equivariant isomorphisms 
$$
p_\G: \Gr^W_\bullet\u_\G\cong\Gr^W_\bullet\u_g\,\,\,\,\text{ and }\,\,\,\, p_{\G'}:\Gr^W_\bullet\u_{\G'}\cong\Gr^W_\bullet\u_{g,1}.
$$
Denote the continuous relative completions of $\pi_1^\alg(\M_\G)\cong \widehat{\G}$ and $\pi_1^\alg(\M_{\G'})\cong \widehat{\G'}$ with respect to their representations to $\Sp(H_\Ql)$ by $\cG_{\G/\Ql}$ and $\cG_{\G'/\Ql}$ and their unipotent radicals by $\U_{\G/\Ql}$ and $\U_{\G'/\Ql}$, respectively. The Lie algebras of $\cG_{\G/\Ql}$, $\cG_{\G'/\Ql}$, $\U_{\G/\Ql}$, and $\U_{\G'/\Ql}$ are denoted by $\g_{\G/\Ql}$, $\g_{\G'/\Ql}$, $\u_{\G/\Ql}$, and $\u_{\G'/\Ql}$, respectively. By Theorem \ref{cont rel comp iso}, there are isomorphisms
$$
\cG_{\G/\Ql}\cong \cG_{\G}\otimes_\Q \Ql\,\,\,\,\text{ and } \,\,\,\,\cG_{\G'/\Ql}\cong \cG_{\G'}\otimes_\Q \Ql,
$$
and so the isomorphisms $p_\G$ and $p_{\G'}$ induce corresponding $\Sp(H_\Ql)$-equivariant isomorphisms
$$
p_{\G/\Ql}: \Gr^W_\bullet\u_{\G/\Ql}\cong\Gr^W_\bullet\u_{g/\Ql}\,\,\,\,\text{ and }\,\,\,\, p_{\G'/\Ql}:\Gr^W_\bullet\u_{\G'/\Ql}\cong\Gr^W_\bullet\u_{g,1/\Ql}.
$$

Now, there exists a commutative diagram:
$$
\xymatrix@R=1em@C=2em{
0\ar[r]&\Gr^W_\bullet\p_\Ql\ar[r]\ar@{=}[d]&\Gr^W_\bullet\u_{\G'/\Ql}\ar[r]\ar[d]^{p_{\G'/\Ql}}&\Gr^W_\bullet\u_{\G/\Ql}\ar[r]\ar[d]^{p_{\G/\Ql}}&0\\
0\ar[r]&\Gr^W_\bullet\p_\Ql\ar[r]&\Gr^W_\bullet\u_{g,1/\Ql}\ar[r]&\Gr^W_\bullet\u_{g/\Ql}\ar[r]&0.
}
$$
Let $s$ be a continuous section of the projection $\pi_1^\alg(\M_{\G'})\to \pi_1^\alg(\M_{\G})$. Then $s$ induces an $\Sp(H_\Ql)$-equivariant graded Lie algebra section of $\Gr^W_\bullet\u_{\G'/\Ql}\to \Gr^W_\bullet\u_{\G/\Ql}$, and hence yields a section of $\Gr^W_\bullet\u_{g,1/\Ql}\to\Gr^W_\bullet\u_{g/\Ql}$. By Theorem \ref{no sec for u}, there is no such section. Therefore, $s$ does not exist. 

\subsection{Proof of Theorem \ref{main thm for tg}}
In this section, we will prove the case over $\Ql$, from which the primary case over $\Q$ will follow. 
\begin{theorem}\label{main thm for t over ql} 
If $g\geq 3$, the projection $\Gr^W_\bullet\t_{g,1/\Ql} \to \Gr^W_\bullet\t_{g/\Ql}$ has no $\Sp(H_\Ql)$-equivariant graded Lie algebra section. 
\end{theorem}
\begin{proof}
By Theorem \ref{t to u mor of mhs}, the Lie algebra map $d\tilde{j}_{g,n}:\t_{g,n}\to \u_{g,n}$ induced by the inclusion $j_{g,n}:T_{g,n}\to \G_{g,n}$ is a morphism of MHSs and it is $\G_{g,n}$-equivariant via conjugation action. By Proposition \ref{t to u sp}, the map $d\tilde{j}_{g,n}$ induces an $\Sp(H)$-equivariant graded Lie algebra homomorphism $\Gr^W_\bullet d\tilde{j}_{g,n}:\Gr^W_\bullet\t_{g,n}\to \Gr^W_\bullet\u_{g,n}$. \textcolor{black}{By Proposition \ref{weight agrees lcs}}, the weight filtration of $\t_{g,n}$ agrees with its lower central series $L^\bullet\t_{g,n}$, and so it is defined over $\Q$.
By base change to $\Ql$, the weight filtrations of $\t_{g,n}$ and $\u_{g,n}$ lift to those of $\t_{g,n/\Ql}$ and $\u_{g,n/\Ql}$, and the induced map $d\tilde{j}_{g,n/\Ql}:\t_{g,n/\Ql}\to \u_{g,n/\Ql}$ preserves the weight filtrations and induces an $\Sp(H_\Ql)$-equivariant graded Lie algebra map $\Gr^W_\bullet d\tilde{j}_{g,n/\Ql}: \Gr^W_\bullet\t_{g,n/\Ql}\to \Gr^W_\bullet\u_{g,n/\Ql}$. 
It follows from  \cite[Theorem~3.4]{hain_inf} that the diagram
$$
\xymatrix@R=1em@C=2em{
\t_{g,n}\ar[r]\ar[d]&\u_{g,n}\ar[d]\\
\t_{g}\ar[r]&\u_g
}
$$
is a pullback square. This induces the diagram:
$$
\xymatrix@R=2em@C=2em{
&&0\ar[d]&0\ar[d]\\
&&\Gr^W_\bullet\p_\Ql\ar@{=}[r]\ar[d]&\Gr^W_\bullet\p_\Ql\ar[d]&\\
0\ar[r]&\Ql(1)\ar[r]\ar@{=}[d]&\Gr^W_\bullet\t_{g,1/\Ql}\ar[r]^<{\Gr^W_\bullet d\tilde{j}_{g,1/\Ql}}\ar[d]&\Gr^W_\bullet\u_{g,1/\Ql}\ar[d]\ar[r]&0\\
0\ar[r]&\Ql(1)\ar[r]&\Gr^W_\bullet\t_{g/\Ql}\ar[r]\ar[d]\ar@/_/[u]_s
&\Gr^W_\bullet\u_{g/\Ql}\ar[r]\ar[d]&0\\
&&0&0.
}
$$
Let $s$ be an $\Sp(H_\Ql)$-equivariant graded Lie algebra section of $\Gr^W_\bullet\t_{g,1/\Ql}\to \Gr^W_\bullet\t_{g/\Ql}$. Then the image in $\Gr^W_{-2}\u_{g,1/\Ql}$ of the trivial representation $\Ql(1)$ under the composition $\Gr^W_{-2}d\tilde{j}_{g,1/\Ql}\circ s$  lies in $\Gr^W_{-2}\p_{\Ql}$. Since $\Gr^W_{-2}\p_{\Ql}$ does not contain a copy of $\Ql(1)$ as an $\Sp(H_\Ql)$-representation by Proposition \ref{sp str of p}, the composition $\Gr^W_\bullet d\tilde{j}_{g,1/\Ql}\circ s$ vanishes on $\Ql(1)$. Thus, $s$ induces an $\Sp(H_\Ql)$-equivariant graded Lie algebra section $\bar s$ of $\Gr^W_\bullet\u_{g,1/\Ql}\to \Gr^W_\bullet\u_{g/\Ql}$, which again cannot exist by Theorem \ref{no sec for u}.
Therefore, the section $s$ cannot exist. 
\end{proof}

As to Corollary \ref{no sp on ab of tg}, we will first consider the profinite case. Each section of the projection $\widehat{T_{g,1}}\to \widehat{T_g}$ induces a continuous section of $\widehat{T_{g,1}}^\un_{/\Ql}\to \widehat{T_{g}}^\un_{/\Ql}$. Since $H_1(\widehat{T_{g,n}}^\un_{/\Ql})=H_1(T_{g,n};\Ql)$, it yields a map $H_1(T_{g};\Ql)\to H_1(T_{g,1};\Ql)$.
Suppose that $g\geq 3$ and $\widehat{T_{g,1}}\to \widehat{T_g}$ admits a continuous section $s$ that induces an $\Sp(H_\Ql)$-equivariant map $s^\ab_\Ql:H_1(T_{g};\Ql)\to H_1(T_{g,1};\Ql)$. 
Recall that the weight filtration of $\t_{g,n}$ lifts to that of $\t_{g,n/\Ql}$, and so 
we have the isomorphism of graded Lie algebras
$$
\Gr^W_\bullet\t_{g,n/\Ql}\cong \Gr^\bullet_L\t_{g,n/\Ql}.
$$
The section $s$ induces a Lie algebra section $ds$ of $\t_{g,1/\Ql}\to \t_{g/\Ql}$ and hence a graded Lie algebra section $\Gr^\bullet_L ds$ of $\Gr^\bullet_L\t_{g,1/\Ql}\to \Gr^\bullet_L\t_{g/\Ql}$. Each $\Gr^m_Lds$ is determined by $\Gr^1_Lds =s^\ab_\Ql$, and so is $\Sp(H_\Ql)$-equivariant. Therefore, $\Gr^\bullet_Lds$ is an $\Sp(H_\Ql)$-equivariant graded Lie algebra section, which cannot exist by Theorem \ref{main thm for t over ql}. \\
\indent Now, suppose that $s$ is a section of $T_{g,1}\to T_g$ inducing an $\Sp(H)$-equivariant map $s^\ab:H_1(T_g;\Q)\to H_1(T_{g,1};\Q)$.  Taking profinite completion yields a continuous section of $\widehat{T_{g,1}}\to \widehat{T_g}$, which in turn yields a continuous section $s^\un$ of $\widehat{T_{g,1}}^\un_{/\Ql}\to \widehat{T_{g}}^\un_{/\Ql}$.  Now, note that  $s^\un$ induces a map on the abelianization, which agrees with $s^\ab\otimes \Ql$, and therefore is $\Sp(H_\Ql)$-equivariant, which contradicts Theorem \ref{main thm for t over ql}.




\end{document}